\documentclass{amsart}
\usepackage{amssymb}
\usepackage{showlabels}

\newcommand{\Q}{\mathbb{Q}}
\newcommand{\C}{\mathbb{C}}

\newcommand{\R}{\mathbb{R}}
\newcommand{\N}{\mathbb{N}}

\newcommand{\dom}{\operatorname{dom}}
\newcommand{\ran}{\operatorname{ran}}

\newcommand{\norm}[1]{\left\| #1 \right\|}
\newcommand{\supp}{\operatorname{supp}}

\newtheorem{theorem}{Theorem}[section]
\newtheorem{lemma}[theorem]{Lemma}

\theoremstyle{definition}
\newtheorem{definition}[theorem]{Definition}

\theoremstyle{theorem}
\newtheorem{corollary}[theorem]{Corollary}

\theoremstyle{theorem}
\newtheorem{proposition}[theorem]{Proposition}

\theoremstyle{theorem}

\theoremstyle{theorem}

\theoremstyle{definition}

\theoremstyle{theorem}

\numberwithin{equation}{section}

\begin{document}
\title{The isometry degree of a computable copy of $\ell^p$}
\author{Timothy H. McNicholl}
\address{Department of Mathematics\\
Iowa State University\\
Ames, Iowa 50011}
\email{mcnichol@iastate.edu}

\author{D. M. Stull}
\address{Department of Computer Science\\
Iowa State University\\
Ames, Iowa 50011}
\email{dstull@iastate.edu}
\thanks{Research of the first author supported in part by a Simons Foundation grant \# 317870.  Research of the second author supported in part by National Science Foundation Grants 1247051 and 1545028.}
\begin{abstract}
When $p$ is a computable real so that $p \geq 1$, we define the isometry degree of a computable presentation of $\ell^p$ to be the least powerful Turing degree $\mathbf{d}$ by which it is $\mathbf{d}$-computably isometrically isomorphic to the standard presentation of $\ell^p$.  
We show that this degree always exists and that when $p \neq 2$ these degrees are precisely the c.e. degrees.
\end{abstract}
\maketitle

\section{Introduction}\label{sec:intro}

Complexity of isomorphisms is a recurring theme of computable structure theory.  
For example, a computably presentable structure is \emph{computably categorical} if 
there is a computable isomorphism between any two of its computable presentations; 
it is \emph{$\Delta_n^0$-categorical} if there is a $\Delta_n^0$ isomorphism between any two of 
its computable copies.  The \emph{degree of categoricity} of a computable structure
is the least powerful oracle that computes an isomorphism between any two of its computable copies \cite{Fokina.Kalimullin.Miller.2010}.  There is at this time no characterization of the degrees of categoricity.  Partial results can be found in \cite{Anderson.Csima.2016}, \cite{Csima.Franklin.Shore.2013}, and \cite{Fokina.Kalimullin.Miller.2010}.

Throughout most of its development, computable structure theory has focused on countable structures. 
However, there has recently emerged a program to apply the concepts of computable structure theory to the uncountable structures commonly encountered in analysis such as metric spaces and Banach spaces.  For example, A.G. Melnikov has shown that $C[0,1]$ is not computably categorical as a metric space \cite{Melnikov.2013}, and  Melnikov and Ng have shown that $C[0,1]$ is not computably categorical as a Banach space \cite{Melnikov.Ng.2014}.  In their seminal text, Pour-El and Richards proved that $\ell^1$ is not computably categorical and that $\ell^2$ is computably categorical (though the results were not framed in the language of computable structure theory) \cite{Pour-El.Richards.1989}.  In 2013 Melnikov asked if $\ell^p$ is computably categorical for any values of $p$ besides $2$ \cite{Melnikov.2013}.  In 2015, the first author answered this question in the negative and later showed that $\ell^p$ is $\Delta^0_2$-categorical whenever $p$ is a 
computable real so that $p \geq 1$ \cite{McNicholl.2017}, \cite{McNicholl.2015}.

Here we put forward the study of a new notion: the \emph{degree of isomorphism} for a pair $(\mathcal{A}^\#, \mathcal{A}^+)$ of computable presentations of a structure $\mathcal{A}$; this is defined to be the least powerful oracle that computes an isomorphism of $\mathcal{A}^\#$ onto $\mathcal{A}^+$.  This notion fits in with the general theme of studying complexity of isomorphisms and is a local version of the concept of degree of categoricity.   
If among all computable presentations of $\mathcal{A}$ one is regarded as standard, then we define the isomorphism degree of a single computable presentation $\mathcal{A}^\#$ of $\mathcal{A}$ to be the least powerful oracle that computes an isomorphism of the standard presentation with $\mathcal{A}^\#$.  

We propose to study degrees of isomorphism in the context of the new intersection of computable
structure theory and computable analysis, specifically with regard to computable copies of $\ell^p$.  
So, whenever $(\ell^p)^\#$ is a computable presentation of $\ell^p$, we define the isometry degree of $(\ell^p)^\#$ to be the least powerful Turing degree that computes a linear isometry of the standard presentation of $\ell^p$ onto $(\ell^p)^\#$.

It is not obvious that degrees of isomorphism always exist.  For example, R. Miller has produced a computable structure with no degree of computable categoricity \cite{Miller.2009}.  We are thus pleasantly surprised to find that computable presentations of $\ell^p$ always have an isometry degree and that we can say precisely what these degrees are.  Specifically, we prove the following two theorems.

\begin{theorem}\label{thm:main.1}
When $p$ is a computable real so that $p \geq 1$, 
every computable presentation of $\ell^p$ has a degree of isometry, and this degree is c.e..
\end{theorem}

\begin{theorem}\label{thm:main.2}
When $p$ is a computable real so that $p \geq 1$ and $p \neq 2$, 
the isometry degrees of the computable presentations of $\ell^p$ 
are precisely the c.e. degrees.
\end{theorem}

One direction of Theorem \ref{thm:main.2} is already known; namely that every c.e. degree is an isometry degree  \cite{McNicholl.2015}.  However, we give a new proof which we believe is simpler and more intuitive. 

The paper is organized as follows.  Sections \ref{sec:back} and \ref{sec:prelim} cover background and preliminaries 
from functional analysis and computable analysis.  Section \ref{sec:compression} contains a required result on the complexity of uniformly right-c.e. sequences of reals which is perhaps interesting in its own right.  Section \ref{sec:proof.thm.2.part.1} contains the new proof that, when $p \neq 2$, every c.e. degree is the isometry degree of a computable presentation of $\ell^p$.  In Section \ref{sec:proof.thm.1}, we show 
that every computable presentation of $\ell^p$ has a degree of linear isometry and that this degree is c.e..

\section{Background}\label{sec:back}

Let $\N = \{0,1, \ldots \}$.  

\subsection{Arboreal matters}\label{sec:back::subsec:arb}

Let $\N^*$ denote the set of all finite sequence of natural numbers.  Note that $\N^*$ contains the empty sequence $\emptyset$.  When $\nu \in \N^*$, let $|\nu|$ denote its length; i.e. the cardinality of its domain.  When $\nu, \nu' \in \N^*$, we write $\nu \subset \nu'$ to mean that  
$\nu$ is a prefix of $\nu'$; for, in this case, since $\nu$ and $\nu'$ are sets of ordered pairs, it indeed is equivalent to say that $\nu$ is a proper subset of $\nu'$.  When $\nu \subset \nu'$, we say that 
$\nu$ is an \emph{ancestor} of $\nu'$.  The maximal ancestor of a nonempty $\nu \in \N^*$ is its \emph{parent}.  
If $\nu$ is the parent of $\nu'$, then we say $\nu'$ is a \emph{child} of $\nu$.  We let $\nu^-$ denote the parent of $\nu$.

A \emph{tree} is a subset $S$ of $\N^*$ that is closed under ancestors; that is, whenever $\nu \in S$, every ancestor of $\nu$ is in $S$.  Suppose $S$ is a tree.  When $\nu \in S$ we refer to $\nu$ as a \emph{node of $S$}.  Thus, $\emptyset$ is a node of every tree; we refer to $\emptyset$ as the \emph{root node}.  
We say a node $\nu$ of $S$ is \emph{terminal} if none of its children belong to $S$.    
Finally, we say that a function $f : S \rightarrow \R$ is \emph{decreasing} if $f(\nu) > f(\nu')$ whenever 
$\nu' \in S$ and $\nu \subset \nu'$.  

\subsection{Background from functional analysis}\label{sec:back::subsec:FA}

We assume that the field of scalars is the complex numbers although all results hold when the field of scalars is the real numbers.  A scalar is \emph{unimodular} if $|\lambda| = 1$.  Let $D(z;r)$ denote the open disk whose center is $z$ and whose radius is $r$.  

Recall that a Banach space is a complete normed linear space.  A subset of a Banach space $\mathcal{B}$ is \emph{linearly dense} if its linear span is dense in $\mathcal{B}$.

The simplest example of a Banach space is $\C^n$ where the norm is given by 
\[
\norm{(z_1, \ldots, z_n)} = \sqrt{\sum_{j=1}^n |z_j|^2 }.
\]

Suppose $1 \leq p < \infty$.  Recall that $\ell^p$ is the set of all functions $f : \N \rightarrow \C$ so that 
$\sum_{n = 0}^\infty |f(n)|^p < \infty$.  When $f \in \ell^p$, the $\ell^p$-norm of $f$ is defined to be 
\[
\norm{f}_p = \left( \sum_{n = 0}^\infty |f(n)|^p \right)^{1/p}.
\]
It is well-known that $\ell^p$ is a Banach space.  For each $n \in \N$, let $e_n = \chi_{\{n\}}$.  
Then, $\{e_0, e_1, \ldots\}$ is the standard basis for $\ell^p$.

Suppose that $\mathcal{B}_0$ and $\mathcal{B}_1$ are Banach spaces and that $T : \mathcal{B}_0 \rightarrow \mathcal{B}_1$.  
If there is a constant $C > 0$ so that $\norm{T(v)}_{\mathcal{B}_1} \leq C \norm{v}_{\mathcal{B}_0}$ for all $v \in \mathcal{B}_0$, then $T$ is \emph{bounded}.  If $T$ is linear, then $T$ is continuous if and only if $T$ is bounded.  
$T$ is an \emph{isomorphism} if it is a linear homeomorphism.  $T$ is \emph{isometric} if 
$\norm{T(u) - T(v)}_{\mathcal{B}_1} = \norm{u - v}_{\mathcal{B}_0}$ whenever $u, v \in \mathcal{B}_0$.  
An isometric isomorphism thus preserves the linear and metric structure of the Banach spaces.   
Finally, if $\mathcal{B}_1 = \C$, then $T$ is a \emph{functional}.

Suppose $1 \leq p < \infty$ and $\frac{1}{p} + \frac{1}{q} = 1$ (i.e. $q$ is the \emph{conjugate} of $p$).  When $f \in \ell^p$ and $g \in \ell^q$, let 
\[
\langle f,g \rangle = \sum_{n = 0}^\infty f(n) \overline{g(n)}.
\]

When $f \in \ell^q$, let $f^*(g) = \langle g,f \rangle$ for all $g \in \ell^p$.  By H\"older's inequality, $\norm{f^*(g)}_1 \leq \norm{g}_p \norm{f}_q$.  Thus, $|f^*(g)| \leq \norm{g}_p\norm{f}_q$, and so $f^*$ is a bounded linear functional on $\ell^p$.

When $f \in \ell^p$, the \emph{support} of $f$, which we denote by $\supp(f)$, is the set of all $n \in \N$ so that 
$f(n) \neq 0$.  Vectors $f,g \in \ell^p$ are \emph{disjointly supported} if their supports are disjoint.  
A subset of $\ell^p$ is disjointly supported if any two of its elements are disjointly supported.  
We will make frequent use of the following observation: if $f_0, \ldots, f_n \in \ell^p$ are disjointly supported, then $\norm{f_0 + \ldots + f_n}_p^p = \norm{f_0}_p^p + \ldots + \norm{f_n}_p^p$.  

We will make use of the following, which is fairly well-known and has a straightforward proof, to construct linear isometries.   

\begin{proposition}\label{prop:unique.linear}
Suppose $1 \leq p < \infty$ and $\{g_n\}_{n \in \N}$ is a sequence of nonzero disjointly supported vectors of $\ell^p$.  Then, there is a unique linear isometry $T : \ell^p \rightarrow \ell^p$ so that 
$T(e_n) = \norm{g_n}^{-1} g_n$.  
\end{proposition}

When $f,g \in \ell^p$, let $\sigma_0(f,g) = |2(\norm{f}_p^p + \norm{g}_p^p) - \norm{f + g}_p^p - \norm{f - g}_p^p|$.  The following was proven in 1956 by O. Hanner and independently by J. Lamperti in 1958 \cite{Hanner.1956}, \cite{Lamperti.1958}.

\begin{proposition}\label{prop:sigma}
Suppose $1 \leq p < \infty$ and $p \neq 2$.  Then, $f,g \in \ell^p$ are disjointly supported if and only if
$\sigma_0(f,g) = 0$.
\end{proposition}

The following are more or less immediate consequences of Proposition \ref{prop:sigma}.  They were first observed by S. Banach and later rigorously proven by J. Lamperti \cite{Banach.1987}, \cite{Lamperti.1958}.

\begin{theorem}\label{thm:pres.disj.support}
Suppose $1 \leq p < \infty$ and $p \neq 2$.  If $T : \ell^p \rightarrow \ell^p$ is linear and isometric, then 
$T$ preserves disjointness of support.  That is, $T(f)$ and $T(g)$ are disjointly supported whenever $f,g \in \ell^p$ are disjointly supported.
\end{theorem}

\begin{theorem}\label{thm:classification}
Suppose $p$ is a real number so that $p \geq 1$ and $p \neq 2$.  Let $T$ be a linear map of $\ell^p$ onto $\ell^p$.  Then, $T$ is an isometric isomorphism if and only if there is a permutation $\phi$ of $\N$ and a sequence $\{\lambda_n\}_{n \in \N}$ of unimodular scalars so that $T(e_n) = \lambda_n e_{\phi(n)}$ for all $n$.
Furthermore, if $\phi$ is a permutation of $\N$, and if $\Lambda = \{\lambda_n\}_{n \in \N}$ is a sequence of unimodular scalars, then there is a unique isometric isomorphism $T_{\phi, \Lambda}$ of $\ell^p$ so that 
$T_{\phi, \Lambda}(e_n) = \lambda_n e_{\phi(n)}$ for each $n \in \N$.
\end{theorem}

We now summarize some definitions and results from \cite{McNicholl.2017}.  When $f,g \in \ell^p$, write $f \preceq g$ if and only if $f = g \cdot \chi_A$ for some $A \subseteq \N$.  
In this case we say $f$ is a \emph{subvector} of $g$.  It follows that the subvector relation is a partial order on $\ell^p$.  Accordingly, if $\mathcal{B}$ is a subspace of $\ell^p$, then $f \in \mathcal{B}$ is an \emph{atom of $\mathcal{B}$} if there is no $g \in \mathcal{B}$ so that $\mathbf{0} \prec g \prec f$.  It follows that $f$ is an atom of $\ell^p$ if and only if $f$ is a nonzero scalar multiple of a standard basis vector.  

Note that $f$ is a subvector of $g$ if and only if $f$ and $g - f$ are disjointly supported.  Thus, when $p \neq 2$, the subvector ordering of $\ell^p$ is preserved by linear isometries.

Suppose $S$ is a tree and $\phi : S \rightarrow \ell^p$.  We say $\phi$ is \emph{separating} if 
$\phi(\nu)$ and $\phi(\nu')$ are disjointly supported whenever $\nu, \nu' \in S$ are incomparable.  
We say $\phi$ is \emph{summative} if for every nonterminal node $\nu$ of $S$, $\phi(\nu) = \sum_{\nu'} \phi(\nu')$ where $\nu'$ ranges over the children of $\nu$ in $S$.  
Finally, we say $\phi$ is a \emph{disintegration} if it is injective, separating, summative, never zero, and if its range is linearly dense in $\ell^p$.   

Suppose $\phi : S \rightarrow \ell^p$ is a disintegration.  A chain $C \subseteq S$ is \emph{almost norm-maximizing} if whenever $\nu \in C$ is a nonterminal node of $S$, $C$ contains a child $\nu'$ of $\nu$ so that 
\[
\max_\mu \norm{\phi(\mu)}_p^p \leq \norm{\phi(\nu')}_p^p + 2^{-|\nu|}  
\]
where $\mu$ ranges over the children of $\nu$ in $S$.  The existence of such a child follows from calculus.  

The following is proven in \cite{McNicholl.2017}.

\begin{theorem}\label{thm:lim.chains}
Suppose $1 \leq p < \infty$ and $p \neq 2$, and let $\phi$ be a disintegration of $\ell^p$. 
\begin{enumerate}
	\item If $C$ is an almost norm-maximizing chain of $\phi$, then the $\preceq$-infimum of $\phi[C]$ exists and is either \textbf{0} or an atom of $\preceq$.  Furthermore, $\inf\phi[C]$ is the limit in the $\ell^p$ norm of $\phi(\nu)$ as $\nu$ traverses the nodes in $C$ in increasing order.\label{thm:lim.chains::itm:inf}
		
	\item If $\{C_n\}_{n=0}^\infty$ is a partition of $\dom(\phi)$ into almost norm-maximizing chains, then $\inf\phi[C_0], \inf\phi[C_1], ...$ are disjointly supported.  Furthermore, for each $j \in \N$, there exists a unique $n$ so that $\{j\}$ is the support of  $\inf \phi[C_n]$. \label{thm:lim.chains::itm:unique} 
\end{enumerate}
\end{theorem}

\subsection{Background from computable analysis}\label{sec:back::subsec:CA}

We assume the reader is familiar with the central concepts of computability theory, including computable and computability enumerable sets, Turing reducibility, and enumeration reducibility. These are explained in \cite{Cooper.2004}.  We begin with the application of computability concepts to Banach spaces.  Our approach is essentially the same as in \cite{Pour-El.Richards.1989}.

A real $r$ is \emph{left (right)-c.e.} if its left (right) Dedekind cut is c.e..  A sequence $\{r_n\}_{n \in \N}$ of reals is \emph{uniformly left (right)-c.e.} if the left (right) Dedekind cut of $r_n$ is c.e. uniformly in $n$.

Let $\mathcal{B}$ be a Banach space.  A function $R : \N \rightarrow \mathcal{B}$ is a \emph{structure} on $\mathcal{B}$ if its range is linearly dense in $\mathcal{B}$.  If $R$ is a structure on $\mathcal{B}$, then $(\mathcal{B}, R)$ is a \emph{presentation} of $\mathcal{B}$. 

A Banach space may have a presentation that is designated as \emph{standard}; such a space is identified with its standard presentation.  In particular, if we let $R(n) = e_n$, then $(\ell^p, R)$ is the standard presentation of $\ell^p$.  If $R(j)$ is the $(j+1)$st vector in the standard basis for $\C^n$ when $j < n$, and if $R(j) = \mathbf{0}$ when $j \geq n$, then $(\C^n, R)$ is the standard presentation of $\C^n$.

Suppose $\mathcal{B}^\# = (\mathcal{B}, R)$ is a presentation of $\mathcal{B}$.  Then, $\mathcal{B}^\#$ induces associated classes of rational vectors and rational open balls as follows.  We say $v \in \mathcal{B}$ is a \emph{rational vector} of $\mathcal{B}^\#$ if there exist $\alpha_0, \ldots, \alpha_M \in \Q(i)$ so that $v = \sum_{j = 0}^M \alpha_j R(j)$.  A \emph{rational open ball of $\mathcal{B}^\#$} is an open ball whose center is a rational vector of $\mathcal{B}^\#$ and whose radius is a positive rational number.

The rational vectors of $\mathcal{B}^\#$ then give rise to associated classes of computable vectors and sequences.  A vector $v \in \mathcal{B}$ is a \emph{computable vector of $\mathcal{B}^\#$} if there is an algorithm that given any $k \in \N$ as input produces a rational vector $u$ of $\mathcal{B}^\#$ so that $\norm{u - v}_{\mathcal{B}} < 2^{-k}$.  A sequence $\{v_n\}_{n \in \N}$ of vectors of $\mathcal{B}$ is a \emph{computable sequence of $\mathcal{B}^\#$} if $v_n$ is a computable vector of $\mathcal{B}^\#$ uniformly in $n$.

When $X \subseteq \N$, the classes of $X$-computable vectors and $X$-computable sequences of $\mathcal{B}^\#$ are defined by means of the usual relativizations.   If $S \subseteq \N^*$, then the definitions of the classes of computable and $X$-computable maps from $S$ into $\mathcal{B}^\#$ are similar to the definitions of computable and $X$-computable sequences of $\mathcal{B}^\#$. 

Presentations $\mathcal{B}_0^\#$ and $\mathcal{B}_1^\#$ of Banach spaces $\mathcal{B}_0$ and $\mathcal{B}_1$ respectively induce an associated class of computable maps from 
$\mathcal{B}_0^\#$ into $\mathcal{B}_1^\#$.  Namely, a map $T : \mathcal{B}_0 \rightarrow \mathcal{B}_1$ is said to be a \emph{computable map of $\mathcal{B}_0^\#$ into $\mathcal{B}_1^\#$} if 
there is an algorithm $P$ with the following properties:
\begin{enumerate}
	\item Given a (code of a) rational ball $B_1$ of $\mathcal{B}_0^\#$ as input, if $P$ halts then it produces a rational ball $B_2$ of $\mathcal{B}_1^\#$ so that $T[B_1] \subseteq B_2$.  
	
	\item If $U$ is a neighborhood of $T(v)$, then there is a rational ball $B_1$ of $\mathcal{B}_0^\#$ so that $v \in B_1$ and given $B_1$, $P$ produces a rational ball $B_2 \subseteq U$.
\end{enumerate}
 In other words, it is possible to compute arbitrarily good approximations of $T(v)$ from sufficiently good approximations of $v$.  This definition relativizes in the obvious way.

When the map $T$ is linear, the following well-known characterization is useful.

\begin{theorem}\label{thm:comp.map}
Suppose $\mathcal{B}_1^\#$ and $\mathcal{B}_2^\#$ are presentations of Banach spaces and that $\mathcal{B}_1^\#= (\mathcal{B}_1, R_1)$.  Suppose also that $T : \mathcal{B}_1 \rightarrow \mathcal{B}_2$ is linear.  Then, $T$ is an $X$-computable map of $\mathcal{B}_1^\#$ into $\mathcal{B}_2^\#$ if and only if $\{T(R_1(n))\}_{n \in \N}$ is an $X$-computable sequence of $\mathcal{B}_2^\#$.
\end{theorem}

We say that a presentation $\mathcal{B}^\#$ of a Banach space $\mathcal{B}$ is a \emph{computable presentation} if the norm is a computable map from $\mathcal{B}^\#$ into $\C$.

For a proof of the following see \cite{Ziegler.2006} or Section 6.3 of \cite{Weihrauch.2000}.

\begin{proposition}\label{prop:zero.find}
Suppose $r$ is a computable positive number. 
If $f$ is a computable real-valued function on $\C$, and if $f$ has exactly one zero in $D(0;r)$, then this zero is a computable point.  Furthermore, this zero can be computed uniformly in $f,r$.
\end{proposition}

The following is proven in \cite{McNicholl.2017}.

\begin{theorem}\label{thm:comp.disint}
Suppose $p$ is a computable real so that $p \geq 1$ and $p \neq 2$.  Then, every computable presentation of $\ell^p$ has a computable disintegration.
\end{theorem}

\section{Preliminaries}\label{sec:prelim}

\subsection{Preliminaries from functional analysis}\label{sec:prelim::subsec:FA}

Let $1 \leq p < \infty$, and suppose $f$ is a unit atom of $\ell^p$ (i.e. an atom of norm 1).  Then, $f$ is also a unit vector of $\ell^q$ where $q$ is the conjugate of $p$.  So, $|f^*(g)| \leq \norm{g}_p$.  It also follows that $f^*(g)f \preceq g$ for all $g \in \ell^p$.  Furthermore, if $f^*(g) = 0$, then $f$ and $g$ are disjointly supported.  Finally, if $g$ is an atom of $\ell^p$, and if $f$ and $g$ are not disjointly supported, then $\supp(f) = \supp(g)$ and $f^*(g)f = g$.

Our proof of Theorem \ref{thm:main.2} will utilize the following.

\begin{lemma}\label{lm:recognize}
Suppose $1 \leq p < \infty$, and suppose $\phi : S \rightarrow \ell^p$ is a disintegration of $\ell^p$.  
Let $C \subseteq S$ be a chain so that whenever $\nu \in C$ is a nonterminal node of $S$, $C$ contains a child $\nu'$ of $\nu$ so that 
\[
\max\{\norm{\phi(\mu)}_p^p\ :\ \mu \in \nu_S^+\} - \norm{\phi(\nu')}_p^p < \min\{\frac{1}{2}\norm{\phi(\nu)}_p^p, 2^{-|\nu|}\}.
\]
Suppose $f$ is a unit atom of $\ell^p$.  
\begin{enumerate}
	\item \label{lm:recognize::itm:exists} If $f$ and $\inf \phi[C]$ are not disjointly supported, then there is a $\nu \in C$ so that 
	\begin{equation}
	\norm{\phi(\nu) - f^*(\phi(\nu))f}_p^p + \min\{\frac{1}{2}\norm{\phi(\nu)}_p^p, 2^{-|\nu|}\} < \norm{f^*(\phi(\nu))f}_p^p. \label{ineq:recognize}
	\end{equation}
	
	\item  \label{lm:recognize::itm:valid} If $\nu \in C$ satisfies \normalfont{(\ref{ineq:recognize})}, then $\inf \phi[C] = f^*(\phi(\nu))f$.  
\end{enumerate}
\end{lemma}

\begin{proof}
Let $g = \inf \phi[C]$.  Let $\epsilon(\nu) = \min\{\frac{1}{2}\norm{\phi(\nu)}_p^p, 2^{-|\nu|}\}$.  
Thus, $\epsilon$ is decreasing (i.e. $\nu \subset \nu'$ implies $\epsilon(\nu) \leq \epsilon (\nu')$).  Since $\epsilon(\nu) \leq 2^{-|\nu|}$, $C$ is almost norm-maximizing.  
Therefore, $g$ is either $\mathbf{0}$ or an atom.  \\

\noindent(\ref{lm:recognize::itm:exists}): Suppose $g$ and $f$ are not disjointly supported.  Thus, $g \neq \mathbf{0}$.  Therefore, $g$ is an atom and so $f^*(g)f = g$.  

Suppose $C$ is finite.  It follows that $C$ contains a terminal node $\nu$ of $S$.  By Theorem \ref{thm:lim.chains}, $\phi(\nu) = g$.  Since $\epsilon(\nu) < \norm{\phi(\nu)}_p^p$, it follows that $\nu$ satisfies (\ref{ineq:recognize}).

Now, suppose $C$ is infinite.  By assumption, $\lim_{\nu \in C} \epsilon(\nu) = 0$.  
By Theorem \ref{thm:lim.chains}, $\lim_{\nu \in C} \phi(\nu) = g$ in the $\ell^p$-norm.  
Since $f^*$ is continuous, $\lim_{\nu \in C} f^*(\phi(\nu)) = f^*(g)$.  Thus, 
$\lim_{\nu \in C} f^*(\phi(\nu))f = g$.  The existence of a $\nu \in C$ that satisfies (\ref{ineq:recognize}) follows.\\

\noindent(\ref{lm:recognize::itm:valid}): Suppose $\nu \in C$ satisfies (\ref{ineq:recognize}).  Then, $f^*(\phi(\nu)) \neq 0$.  Let $h = f^*(\phi(\nu))f$.  Thus, $h$ is an atom and $h \preceq \phi(\nu)$.  
Since $h$ is nonzero, it suffices to show that $h \preceq \phi(\mu)$ for all $\mu \in C$.  
By way of contradiction, suppose $h \not \preceq \phi(\mu)$ for some $\mu \in C$.  Hence, $\nu \subset \mu$ and so $\mu^- \in C$.  Without loss of generality, assume $h \preceq \phi(\mu')$ for all $\mu' \subset \mu$.

Since $\phi$ is separating and summative, $h \preceq \phi(\mu')$ for some sibling $\mu'$ of $\mu$.  Therefore, $\norm{h}_p^p \leq \norm{\phi(\mu')}_p^p$.  At the same time, since $\mu^- \in C$, 
$\norm{\phi(\mu')}_p^p \leq \norm{\phi(\mu)}_p^p + \epsilon(\mu^-)$.  Seeing as $\phi$ is separating and summative, $\phi(\mu) \preceq \phi(\mu^-) - h$.  But, as $h \preceq \phi(\mu^-) \preceq \phi(\nu)$, 
$\phi(\mu^-) - h \preceq \phi(\nu) - h$ and so 
$\norm{\phi(\mu^-) - h}_p^p \leq \norm{\phi(\nu) - h}_p^p$.  Since $\epsilon$ is decreasing, 
$\epsilon(\mu^-) \leq \epsilon(\nu)$, and so 
\[
\norm{h}_p^p \leq \norm{\phi(\nu) - h}_p^p + \epsilon(\nu) < \norm{h}_p^p
\]
which is a contradiction.
\end{proof}

\subsection{Preliminaries from computable analysis}\label{sec:prelim::subsec:CA}

We first extend some of the results in \cite{McNicholl.2017} on partitioning the domain of a disintegration into almost norm-maximizing chains.

\begin{lemma}\label{lm:comp.child}
Suppose $p \geq 1$ is computable and that $(\ell^p)^\#$ is a computable presenttion of $\ell^p$.  
Suppose also that $\phi$ is a computable disintegration of $(\ell^p)^\#$.  Then, from a nonterminal node $\nu$ of $\dom(\phi)$ and a positive rational number $\epsilon$ it is possible to compute a child $\nu'$ of $\nu$ in $\dom(\phi)$ so that 
\[
\max_\mu \norm{\phi(\mu)}_p^p - \norm{\phi(\nu')}_p^p < \epsilon
\]
where $\mu$ ranges over all children of $\nu$ in $\dom(\phi)$.
\end{lemma}

\begin{proof}
Let $S = \dom(\phi)$.  Since $\phi$ is computable, $S$ is c.e..  For each $s$, let $\nu^+[s]$ denote the 
set of children of $\nu$ that have been enumerated into $S$ by the end of stage $s$.  

Wait until a child $\nu_0'$ of $\nu$ in $S$ is enumerated.  Then, wait for a stage $s$ so that 
\[
\norm{\phi(\nu_0')}_p^p  > \norm{\phi(\nu)}_p^p - \sum_{\mu \in \nu^+[s]} \norm{\phi(\mu)}_p^p
\]
for some child $\nu_0'$ of $\nu$ in $S$.  As $\phi$ is summative, $\norm{\phi(\nu_0')}_p^p > \norm{\phi(\mu)}_p^p$ whenever 
$\mu$ is a child of $\nu$ in $S$ so that $\mu \not \in \nu^+[s]$.  We then compute and output a $\nu' \in \nu^+[s]$ so that $\norm{\phi(\nu')}_p^p + \epsilon > \norm{\phi(\mu)}_p^p$ for all $\mu \in \nu^+[s]$.  
\end{proof}

\begin{theorem}\label{thm:comp.partition}
Suppose $p \geq 1$ is computable and let $(\ell^p)^\#$ be a computable presentation of $\ell^p$.  
Suppose also that $\phi$ is a computable disintegration of $(\ell^p)^\#$ and that $\epsilon : \dom(\phi) \rightarrow (0, \infty)$ is lower semicomputable.  Then, there is a partition $\{C_n\}_{n \in \N}$ of $\dom(\phi)$ into uniformly c.e. chains so that whenever $\nu \in C_n$ is a nonterminal node of $\dom(\phi)$, $C_n$ contains a child $\nu'$ of $\nu$ so that 
\[
\max_\mu \norm{\phi(\mu)}_p^p - \norm{\phi(\nu')}_p^p < \epsilon(\nu)
\]
where $\mu$ ranges over all children of $\nu$ in $\dom(\phi)$.
\end{theorem}

\begin{proof}
Let $S = \dom(\phi)$.  We define a function $\psi : S \rightarrow S$ as follows.  
By Lemma \ref{lm:comp.child}, from a nonterminal node $\nu$ of $S$ it is possible to compute a child $\nu'$ of $\nu$ in $S$ so that  
\[
\max_\mu \norm{\phi(\mu)}_p^p - \norm{\phi(\nu')}_p^p < \epsilon(\nu)
\]
where $\mu$ ranges over all children of $\nu$ in $\dom(\phi)$; let $\psi(\nu) = \nu'$.  
Then, the orbits of $\psi$ form a decomposition of $S$ into chains with the required properties.  
(Recall that an \emph{orbit} of a function $f : X \rightarrow X$ is a set of the form 
$\{f^n(x_0)\ :\ n \in \N\}$.). 
Let 
\[
U = \{\emptyset\}\ \cup\ \{\nu \in S - \{\emptyset\}\ :\ \nu \neq \psi(\nu^-)\}.
\]
Then, $U$ is computable.  Let $\{\upsilon_n\}_{n \in \N}$ be an effective enumeration of $U$.  
Let $C_n$ be the $\psi$-orbit of $\upsilon_n$.  It follows that $\{C_n\}_{n \in \N}$ is a one-to-one enumeration of the orbits of $\psi$ and that $C_n$ is c.e. uniformly in $n$.
\end{proof}

The proof of Theorem \ref{thm:main.2} will utilize the following.

\begin{proposition}\label{prop:f*.comp} 
Suppose $p$ is a computable real so that $p \geq 1$, and let $(\ell^p)^\#$ be a computable presentation of $\ell^p$.  Suppose $f$ is a unit atom of $\ell^p$.  If $f$ is a computable vector of $(\ell^p)^\#$, then 
$f^*$ is a computable functional of $(\ell^p)^\#$.
\end{proposition}

\begin{proof}
Suppose $p = 2$.  Thus, $p$ is its own conjugate.  Since $f$ is a computable vector of $(\ell^p)^\#$, it follows from the polar identity that $f^*$ is a computable functional on $(\ell^p)^\#$.  

Suppose $p \neq 2$.  Let $(\ell^p)^\# = (\ell^p, R)$.  By Theorem \ref{thm:comp.map}, it suffices to show that $\{f^*(R(j))\}_{j \in \N}$ is a computable sequence of scalars.  Let $j,k \in \N$ be given as input.  Compute approximations of $\norm{R(j)}_p$ until 
it is witnessed that $\norm{R(j)}_p > 0$ or it is witnessed that $\norm{R(j)}_p < 2^{-k}$.  In the latter case, since $|f^*(R(j))| \leq \norm{R(j)}_p$, we can output $0$.  Suppose it is witnessed that 
$\norm{R(j)}_p > 0$.  Let $\phi(\lambda) = \sigma_0(f, R(j) - \lambda f)$ for each $\lambda \in \C$.  
It then follows from the remarks in Section \ref{sec:prelim::subsec:FA} that $f^*(R(j))$ is the unique scalar $\lambda$ so that $\phi(\lambda) = 0$ and that the modulus of this scalar is no larger than $\norm{R(j)}_p$.  We can then deduce from Proposition \ref{prop:zero.find} that it is now possible to compute a rational point $\hat{\lambda}$ so that 
$D(\hat{\lambda}; 2^{-k})$ contains a zero of $\phi$.  So, we output $\hat{\lambda}$.  In either case, we have
computed a rational point that is less than $2^{-k}$ from $f^*(R(j))$.  Hence, $\{f^*(R(j))\}_{j \in \N}$ is computable.
\end{proof}

\section{A compression theorem}\label{sec:compression}

Our proof of Theorem \ref{thm:main.2} will utilize the following theorem which we believe is interesting in its own right.  Roughly speaking, it gives conditions under which the information in a sequence of reals can be compressed into a single real.

\begin{theorem}\label{thm:compression}  Let $\{r_n\}_{n \in \N}$ be a sequence of real numbers.
\begin{enumerate}
	\item If $\{r_n\}_{n \in \N}$ is uniformly right-c.e., then there is a right-c.e. real $r$ so that the join of the left Dedekind cuts of the $r_n$'s is enumeration-equivalent to the left Dedekind cut of $r$.  \label{thm:compression::itm:right.ce}
	
	\item If $\{r_n\}_{n \in \N}$ is uniformly left-c.e., then there is a left-c.e. real $r$ so that the join of the right Dedekind cuts of the $r_n$'s is enumeration-equivalent to the right Dedekind cut of $r$. \label{thm:compression::itm:left.ce}
\end{enumerate}
\end{theorem}

Our proof of Theorem \ref{thm:compression} will employ the following definition.

\begin{definition}\label{def:mod.summ}
Suppose $\{r_n\}_{n =0}^\infty$ is a sequence of real numbers.
A \emph{modulus of summability} for $\{r_n\}_{n \in \N}$ is a function $f : \N \rightarrow \N$
so that $\left| \sum_{n = N_0}^\infty r_n \right| < 2^{-k}$ whenever $k \in \N$ and $N_0 \geq f(k)$.
\end{definition}

We note that if a sequence of reals has a modulus of summability, then its tails form a Cauchy sequence and so its partial sums form a Cauchy sequence; thus, it is summable.  

We now come to our first step toward proving Theorem \ref{thm:compression}.

\begin{proposition}\label{prop:red.to.term}
Suppose $f$ is a computable modulus of summability for $\{r_n\}_{n \in \N}$.  
\begin{enumerate}
	\item The left Dedekind cut of $\sum_{n = 0}^\infty r_n$ is enumeration-reducible to the join of the left Dedekind cuts of the $r_n$'s.  \label{prop:red.to.term::itm:left}
	
	\item The right Dedekind cut of $\sum_{n = 0}^\infty r_n$ is enumeration-reducible to the join of the right Dedekind cuts of the $r_n$'s. \label{prop:red.to.term::itm:right}
\end{enumerate}
\end{proposition}

\begin{proof}
Let $r = \sum_{n = 0}^\infty r_n$. 

Given an enumeration of the left Dedekind cuts of the $r_n$'s, we can compute an enumeration of the 
left Dedekind cut of $\sum_{n = 0}^{N_0} r_n$ uniformly in $N_0$.  Begin cycling through all rational numbers and all pairs of natural numbers.  Whenever $R \in \Q$ and $N_0, k \in \N$ are found so that $R < \sum_{n = 0}^{N_0} r_n$ and $N_0 \geq f(k)$, begin enumerating 
all rational numbers smaller than $R - 2^{-k}$.  Every rational number thus enumerated is smaller than 
$r$.  Suppose $q < r$.  Choose $k$ so that $2^{-k} < \frac{1}{2}(r - q)$.  Choose $N_0$ so that 
$N_0 \geq f(k)$ and so that $\sum_{n = 0}^{N_0} r_n > \frac{1}{2}(r + q)$.  Then, $q < \sum_{n = 0}^{N_0} r_n - 2^{-k}$, and so $q < R - 2^{-k}$ whenever $R$ is a number in $(q + 2^{-k}, \sum_{n = 0}^{N_0} r_n)$.  It follows that every number in the left Dedekind cut of $r$ is enumerated by this process.

Part (\ref{prop:red.to.term::itm:right}) follows from part (\ref{prop:red.to.term::itm:left}).
\end{proof}

\begin{corollary}\label{cor:red.to.sum}
Suppose $f$ is a computable modulus of summability for $\{r_n\}_{n \in \N}$, and let $r = \sum_{n = 0}^\infty r_n$.
\begin{enumerate}
	\item If $\{r_n\}_{n \in \N}$ is uniformly left-c.e., then the right Dedekind cut of $r_n$ is enumeration-reducible to the right Dedekind cut of $r$ uniformly in $n$. \label{cor:red.to.sum::itm:left.right}
	
	\item If $\{r_n\}_{n \in \N}$ is uniformly right-c.e., then the left Dedekind cut of $r_n$ is enumeration-reducible to the left Dedekind cut of $r$ uniformly in $n$.   \label{cor:red.to.sum::itm:right.left} 	
\end{enumerate}
\end{corollary}

\begin{proof}
Suppose $\{r_n\}_{n \in \N}$ is uniformly left-c.e..  Without loss of generality, suppose $n = 0$.  By Proposition \ref{prop:red.to.term}, $r - r_0 = \sum_{n = 1}^\infty r_n$ is left-c.e..  So, since $r_0 = r - (r - r_0)$, from an enumeration of the right Dedekind cut of $r$ we can compute an enumeration of the right Dedekind cut of $r_0$.   Part (\ref{cor:red.to.sum::itm:right.left}) follows from part (\ref{cor:red.to.sum::itm:left.right}).
\end{proof}

\begin{proof}[Proof of Theorem \ref{thm:compression}:]
Suppose $\{r_n\}_{n \in \N}$ is uniformly right-c.e..  We first consider the case where $\{r_n\}_{n \in \N}$ is bounded.   Suppose $M$ is a rational number so that $M > r_n$ for all $n$.  Let $r'_n = 2^{-n}M^{-1} r_n$, and let $f(k) = k+2$.  It follows that $\{r_n'\}_{n = 0}^\infty$ is uniformly right-c.e. and that $f$ is a computable modulus of summability for this sequence.  Let $r = \sum_{n = 0}^\infty r_n'$.  Thus, by Corollary \ref{cor:red.to.sum}, the left Dedekind cut of $r_n'$ is enumeration-reducible to the left Dedekind cut of $r$ uniformly in $n$.  So, the join of these left Dedekind cuts is enumeration reducible to the left Dedekind cut of $r$.  By Proposition \ref{prop:red.to.term}, the left Dedekind cut of $r$ is enumeration-equivalent to the join of the left Dedekind cuts of $r_0', r_1', \ldots$.  Therefore, the left Dedekind cut of $r_n'$ is enumeration-equivalent to the left Dedekind cut of $r_n$ uniformly in $n$. 

If $\{r_n\}_{n \in \N}$ is not bounded, then apply the above procedure to $\{\arctan(r_n)\}_{n \in \N}$.  (Here, we use the fact that $\arctan$ is increasing, bounded, and computable.)

Part (\ref{thm:compression::itm:left.ce}) follows from part (\ref{thm:compression::itm:right.ce}).
\end{proof}

\section{Every c.e. degree is a degree of linear isometry}\label{sec:proof.thm.2.part.1}

Suppose $p$ is a computable real so that $p \geq 1$ and so that $p \neq 2$.  Let $C$ be a c.e. set. Without loss of generality, we can assume $C$ is incomputable.  Let $\{c_n\}_{n \in \mathbb{N}}$ be a one-to-one effective enumeration of $C$. 

For each $n \in \N$, let 
\[
R(n) = \left\{ \begin{array}{cc}
			e_{n} + e_{n + 1} & \mbox{if $n$ even}\\
			e_{2c_{(n-1)/2}} & \mbox{if $n$ odd}
			\end{array}
			\right.
\]
Let $\mathcal{B}$ denote the closed linear span of $\ran(R)$, and let $\mathcal{B}^\# = (\mathcal{B}, R)$.  Since $R$ is a computable sequence of $\ell^p$, it follows that 
$\mathcal{B}^\#$ is a computable presentation of $\mathcal{B}$.   

Note that $e_{2k} + e_{2k+1} \in \mathcal{B}$ for all $k \in \N$ and that 
\[
k \in C\ \Leftrightarrow\ e_{2k} \in \mathcal{B}\ \Leftrightarrow\ e_{2k+1} \in \mathcal{B}.
\]
Note also that if $f$ is an atom of $\mathcal{B}$, then either there exists $k \not \in C$ so that 
$f$ is a nonzero scalar multiple of $e_{2k} + e_{2k+1}$ or there exists $k \in C$ so that $f$ is a nonzero scalar multiple of $e_{2k}$ or $e_{2k+1}$.  

We first claim that $C$ computes an isometric isomorphism of $\ell^p$ onto $\mathcal{B}^\#$.  For, let $\{a_n\}_{n \in \N}$ be the increasing enumeration of $\N - C$.  Let
\begin{eqnarray*}
S(3k) & = & e_{2a_k} + e_{2a_k + 1}\\
S(3k+1) & = & e_{2c_k}\\
S(3k+2) & = & e_{2c_k + 1}
\end{eqnarray*}
Thus, $S$ is a $C$-computable sequence of $\mathcal{B}^\#$.  It also follows that $\ran(S) \subseteq \mathcal{B}$ and that each vector in $\ran(R)$ belongs to the linear span of $\ran(S)$.  Thus, $\ran(S)$ is linearly dense in $\mathcal{B}$.  Since $S$ is a sequence of disjointly supported nonzero vectors, by Proposition \ref{prop:unique.linear}, there is a unique isometric isomorphism $T$ of $\ell^p$ onto $\mathcal{B}$ so that $T(e_n) = \norm{S(n)}_p^{-1} S(n)$ for all $n \in \N$.  By Theorem \ref{thm:comp.map}, $T$ is a $C$-computable map of $\ell^p$ onto $\mathcal{B}^\#$.  

Now, suppose $X \subseteq \N$ computes an isometric isomorphism $T_0$ from $\ell^p$ onto $\mathcal{B}^\#$.  We show that $X$ computes $C$ as follows.  
We first note that since $R$ is a computable sequence of $\ell^p$, $\{T_0(e_j)\}_{j \in \N}$ is an $X$-computable sequence of $\mathcal{B}^\#$.  
We also note that, by the remarks in Section \ref{sec:back::subsec:FA}, $T_0(e_j)$ is a unit atom of the subvector ordering of $\mathcal{B}$.  
Furthermore, if $f$ is a unit atom of the subvector ordering of $\mathcal{B}$, then either $f$ belongs to the subspace generated by $2^{-1/p}(e_{2n} + e_{2n+1})$ for some $n \not \in C$ or $f$ belongs to the subspace generated by $e_{2n + k}$ for some $n \in C$ and $k \leq 1$.  
Also, if $f$ is a unit atom of the subvector ordering of $\mathcal{B}$, then $T_0^{-1}(f)$ is a unit atom of $\ell^p$ and so $f$ belongs to the subspace generated by $T_0(e_j)$ for some $j \in \N$.  

Hence, given $n \in \N$, using oracle $X$ we wait until either $n$ is enumerated into $C$ or a $j \in \N$ is found so that $\min\{\sigma_0(T_0(e_j), e_{2n}), \sigma_0(T_0(e_j), e_{2n+1})\} > 0$.  In the latter case, we know that $2n, 2n+1 \in \supp(T_0(e_j))$ and so $n \not \in C$.  If $n \not \in C$, then $2^{-1/p}(e_{2n} + e_{2n+1})$ is a unit atom of the subvector ordering of $\mathcal{B}$, and so there is a $j \in \N$ so that 
$T_0(e_j)$ is a unimodular scalar multiple of $2^{-1/p}(e_{2n} + e_{2n+1})$.  For this $j$, 
$\min\{\sigma_0(T_0(e_j), e_{2n}), \sigma_0(T_0(e_j), e_{2n+1})\} > 0$.  Thus, this search procedure always terminates.

\section{Every computable copy of $\ell^p$ has a c.e. degree of isometry}\label{sec:proof.thm.1}

Suppose $p \geq 1$ is computable, and let $(\ell^p)^\#$ be a computable presentation of $\ell^p$.  
If $p = 2$, then, as mentioned in the introduction, there is a computable isometric isomorphism of $\ell^p$ onto $(\ell^p)^\#$.  So, suppose $p \neq 2$.  Let $\phi$ be a computable disintegration of $(\ell^p)^\#$, and let $S = \dom(\phi)$.  

For each $\nu \in S$, let $\epsilon(\nu) = \min\{2^{-|\nu|}, \frac{1}{2}\norm{\phi(\nu)}_p^p\}$.  
Thus, $\epsilon$ is computable.  It follows from Theorem \ref{thm:comp.partition} that there is a partition $\{C_n\}_{n \in \N}$ of $S$ into uniformly computable chains so that for every $n$ and every nonterminal $\nu \in C_n$, $C_n$ contains a child $\nu'$ of $\nu$ so that 
\[
\max_\mu \norm{\phi(\mu)}_p^p - \norm{\phi(\nu')}_p^p < \epsilon(\nu)
\]
where $\mu$ ranges over the children of $\nu$ in $S$.  Thus, each $C_n$ is almost norm-maximizing.  Let $g_n = \inf \phi[C_n]$.  

The proof of Theorem \ref{thm:main.1} uses the following lemmas.

\begin{lemma}\label{lm:compute.isom}
If $\{\norm{g_n}_p\}_{n \in \N}$ is an $X$-computable sequence of reals, then $X$ computes an isometric isomorphism of $\ell^p$ onto $(\ell^p)^\#$.  
\end{lemma}

\begin{lemma}\label{lm:compute.seq}
If $X$ computes an isometric isomorphism of $\ell^p$ onto $(\ell^p)^\#$, then $\{\norm{g_n}_p\}_{n = 0}^\infty$ is an $X$-computable sequence of reals.  
\end{lemma}

\begin{proof}[Proof of Lemma \ref{lm:compute.isom}]
Suppose $\{\norm{g_n}_p\}_{n = 0}^\infty$ is an $X$-computable sequence of reals.  

We first claim that $\{g_n\}_{n \in \N}$ is an $X$-computable sequence of $(\ell^p)^\#$.  
For, let $n,k \in \N$ be given.  For each $\nu \in C_n$, $g_n \preceq \phi(\nu)$, and so 
$\norm{\phi(\nu) - g_n}_p^p = \norm{\phi(\nu)}_p^p - \norm{g_n}_p^p$.  Thus, 
for each $\nu \in C_n$, $X$ computes $\norm{\phi(\nu) - g_n}_p$ uniformly in $\nu,n$.  
By Theorem \ref{thm:lim.chains}.\ref{thm:lim.chains::itm:inf}, there is a $\nu \in C_n$ so that $\norm{\phi(\nu) - g_n}_p < 2^{-(k+1)}$; using oracle $X$, such a $\nu$ can be found by a search procedure.  Since $\phi$ is computable, we can additionally compute a rational vector $f$ of $(\ell^p)^\#$ so that $\norm{f - \phi(\nu)}_p < 2^{-(k+1)}$.  
Thus, we have computed a rational vector $f$ of $(\ell^p)^\#$ so that $\norm{f - g_n}_p < 2^{-k}$.  

Let $G$ denote the set of all $n \in \N$ so that $g_n$ is nonzero.  Thus, $G$ is c.e. relative to $X$.
By Theorem \ref{thm:lim.chains}.\ref{thm:lim.chains::itm:unique}, for each $j \in \N$ there is a unique $n \in G$ so that $\supp(g_n) = \{j\}$.  Thus, $G$ is infinite.  
So, $X$ computes a one-to-one enumeration $\{n_k\}_{k  \in \N}$ of $G$.  Let $h_k = \norm{g_{n_k}}_p^{-1} g_{n_k}$.  Thus, $\{h_k\}_{k \in \N}$ is an $X$-computable sequence of $(\ell^p)^\#$.  

Again, by Theorem \ref{thm:lim.chains}.\ref{thm:lim.chains::itm:unique}, for each $j \in \N$, there is a unique $k \in \N$ so that $\supp(h_k) = \{j\}$.  So, there is a permutation $\phi$ of $\N$ so that $\supp(h_k) = \{\phi(k)\}$ for each $k \in \N$.   Since $\norm{h_k}_p = 1$, it follows that there is a unimodular scalar $\lambda_k$ so that $h_k = \lambda_k e_{\phi(k)}$.  
It then follows from Theorem \ref{thm:classification} there there is a unique isometric isomorphism $T$ of $\ell^p$ so that 
$T(e_k) = h_k$ for all $k \in \N$.  By Theorem \ref{thm:comp.map}, $T$ is an $X$-computable map of $\ell^p$ onto $(\ell^p)^\#$.
\end{proof}

\begin{proof}[Proof of Lemma \ref{lm:compute.seq}:]  Set $\epsilon(\nu) = \min\{\frac{1}{2} \norm{\phi(\nu)}_p^p, 2^{-|\nu|}\}$.  
Let $n,k \in \N$ be given.  We compute a rational number $q$ so that 
$|q - \norm{g_n}_p| < 2^{-k}$ as follows.  Using oracle $X$, we search for $\nu \in C_n$ so that either $\norm{\phi(\nu)}_p < 2^{-k}$ or 
so that for some $j \in \N$
\[
\norm{\phi(\nu) - T(e_j)^*(\phi(\nu))T(e_j)}_p^p + \epsilon(\nu) < \norm{T(e_j)^*(\phi(\nu))T(e_j)}_p^p.
\]
By Theorem \ref{thm:classification}, if $g_n \neq \mathbf{0}$, then there exists $j \in \N$ so that 
$T(e_j)$ and $g_n$ have the same support and so $T(e_j)^*(g_n)T(e_j) = g_n$.  
So, by Lemma \ref{lm:recognize}.\ref{lm:recognize::itm:exists}, this search must terminate.  If $\norm{\phi(\nu)}_p < 2^{-k}$, since $g_n \preceq \phi(\nu)$, it follows that $\norm{g_n}_p < 2^{-k}$ and so we output $0$.  Otherwise, it follows from Lemma \ref{lm:recognize}.\ref{lm:recognize::itm:valid} that $T(e_j)^*(\phi(\nu))T(e_j) = g_n$.  Thus, by the relativization of Proposition \ref{prop:f*.comp}, we can use oracle $X$ to compute and output a rational number $q$ so that 
$|q - \norm{T(e_j)^*(\phi(\nu))T(e_j)}_p| < 2^{-k}$.
\end{proof}

Let $r_n = \norm{g_n}_p$.  Since $g_n \preceq \phi(\nu)$ for all $\nu \in C_n$, $r_n \leq \norm{\phi(\nu)}_p$ for all $\nu \in C_n$.  Since $g_n = \inf \phi[C_n]$, it follows from Theorem \ref{thm:lim.chains} that $r_n$ is right-c.e. uniformly in $n$.  So, by Theorem \ref{thm:compression}, there is a right-c.e. real $r$ so that the left Dedekind cut of $r$ is enumeration-equivalent to the join of the left Dedekind cuts of the $r_n$'s.  Let $D$ denote the left Dedekind cut of $r$, and let $\mathbf{d}$ denote the Turing degree of $D$.  Thus, $\mathbf{d}$ is c.e..

We claim that $\mathbf{d}$ is the degree of isometric isomorphism of $(\ell^p)^\#$.  For, since 
$\norm{g_n}_p$ is right-c.e. uniformly in $n$, $\{\norm{g_n}_p\}_{n \in \N}$ is a $D$-computable sequence.  Thus, by Lemma \ref{lm:compute.isom}, $D$ computes an isometric isomorphism of 
$\ell^p$ onto $(\ell^p)^\#$.  Conversely, suppose an oracle $X$ computes an isometric isomorphism of 
$\ell^p$ onto $(\ell^p)^\#$.  It is required to show that $X$ computes $D$.  We can assume $r$ is irrational.  By Lemma \ref{lm:compute.seq}, $X$ computes $\{r_n\}_{n \in \N}$.  
Thus, $X$ computes an enumeration of the uniform join of the left Dedekind cuts of the $r_n$'s.  
Hence, $X$ computes an enumeration of $D$.  
Since $r$ is irrational and right-c.e., it follows that $X$ computes $D$.  

\section{Conclusion}\label{sec:conclusion}

For a computable real $p \geq 1$ with $p \neq 2$, we have investigated the least powerful Turing degree that computes a surjective linear isometry of $\ell^p$ onto one of its computable presentations.  We have shown that this degree always exists, and, somewhat surprisingly, that these degrees are precisely the c.e. degrees.   Thus  computable analysis yields a characterization of the c.e. degrees.   


The isometry degree of a pair of computable copies of $\ell^p$ is an instance of a more general notion of the isomorphism degree of an isomorphic pair of computable structures which is related to the concept of a degree of categoricity.  
Since there exist computable structures for which there is no degree of categoricity, this leads to the question ``Is there a computable structure $\mathcal{A}$ for which there is no degree of computable categoricity but with the property that any two of its computable copies possess a degree of isomorphism?"

\section*{Acknowledgement}
We thank U. Andrews, R. Kuyper, S. Lempp, J. Miller, and M. Soskova for very helpful conversations during the first author's visit to the University of Wisconsin; in particular for suggesting the use of enumeration reducibility.  This visit was funded in part by a travel grant from the Simons Foundation. We also thank Diego Rojas for proofreading and making several very useful suggestions.
Research of the first author supported in part by a Simons Foundation grant \# 317870.  Research of the second author supported in part by National Science Foundation Grants 1247051 and 1545028.

\def\cprime{$'$}
\providecommand{\bysame}{\leavevmode\hbox to3em{\hrulefill}\thinspace}
\providecommand{\MR}{\relax\ifhmode\unskip\space\fi MR }
\providecommand{\MRhref}[2]{%
  \href{http://www.ams.org/mathscinet-getitem?mr=#1}{#2}
}
\providecommand{\href}[2]{#2}

\end{document}